\documentclass{amsart}

\makeatletter
\@namedef{subjclassname@2020}{\textup{2020} Mathematics Subject Classification}
\makeatother

\usepackage{amsthm}
\usepackage{amstext}
\usepackage{enumitem}
\usepackage{amssymb}
\usepackage{adjustbox}
\usepackage{amsmath,calligra,mathrsfs}
\usepackage[all,cmtip]{xy}
\usepackage{dsfont}
\usepackage{hyperref}
\usepackage{graphicx}
\usepackage{colonequals}

\newcommand\bb{\mathbb}

\makeatletter
\renewcommand*\env@matrix[1][*\c@MaxMatrixCols c]{%
  \hskip -\arraycolsep
  \let\@ifnextchar\new@ifnextchar
  \array{#1}}
\makeatother

\theoremstyle{plain}
\newtheorem{thm}{Theorem}[section]

\theoremstyle{definition}

\theoremstyle{remark}
\newtheorem*{rem}{Remark}

\theoremstyle{plain}

\DeclareMathOperator{\Aut}{Aut}

\DeclareMathOperator{\sheafhom}{\mathscr{H}\text{\kern -3pt {\calligra\large om}}\,}

\DeclareMathOperator{\Pic}{Pic}

\DeclareMathOperator{\fake}{fake}

\newcommand{\Mod}[1]{\ (\mathrm{mod}\ #1)}

\newcommand{\FPP}{{\bb P_{\fake}^2}}
\DeclareMathOperator{\rr}{\sqrt{-7}}

\usepackage{lipsum}
\usepackage{fullpage}
\usepackage{amsthm, thmtools}
\usepackage{parskip}

\renewenvironment{proof}{%
    \vspace{-\parskip}\begin{oldproof}%
    }{%
    \end{oldproof}\vspace{-\parskip}%
}

\let\phi\varphi
\usepackage{comment}
\usepackage{xcolor,mathtools}

\title{On the Geometry of a Fake Projective Plane with $21$ Automorphisms}

\author{Lev Borisov}
\address{Hill Center Department of Mathematics, Rutgers University, NJ 08854}
\email{borisov@math.rutgers.edu}
\urladdr{https://sites.math.rutgers.edu/\~{}borisov}

\author{Mattie Ji}
\address{Brown University, Department of Mathematics, Box 1917, 151 Thayer Street, Providence, RI 02912, USA}
\email{mattie\_ji@brown.edu}

\author{Yanxin Li}
\address{Hill Center Department of Mathematics, Rutgers University, NJ 08854}
\email{yl1203@scarletmail.rutgers.edu}

\author{Sargam Mondal}
\address{Middlesex County Academy for Science, Mathematics \& Engineering Technologies}
\email{sm2893@scarletmail.rutgers.edu}

\begin{document}

\begin{abstract}
A fake projective plane is a complex surface with the same Betti numbers as $\mathbb{C} \mathbb P^2$ but not biholomorphic to it. We study the fake projective plane $(a = 7, p = 2, \emptyset, D_3 2_7)$ in the Cartwright-Steger classification. In this paper, we exploit the large symmetries given by $\Aut(\FPP) = C_7 \rtimes C_3$ to construct an embedding of this surface into $\mathbb{C} \mathbb P^5$ as a system of $56$ sextics with coefficients in $\mathbb{Q}(\sqrt{-7})$. For each torsion line bundle $T \in \Pic(\FPP)$, we also compute and study the linear systems $|nH + T|$ with small $n$, where $H$ is an ample generator of the N\'eron-Severi group.
\end{abstract}

\subjclass[2020]{Primary: 14E25, 14Q10. Secondary: 14C20, 14J29.}

\maketitle

\section{Introduction}

 A question by Severi asked whether there exists a complex surface homeomorphic to $\bb C \mathbb P^2$ but not biholomorphic to it. A classical corollary in Yau's proof of the Calabi-Yau conjecture \cite{Yau1977CalabisCA} answered this question in the negative by showing that any complex surface homotopic to $\bb C\mathbb P^2$ must also be biholomorphic to $\bb C\mathbb P^2$. This prompted Mumford \cite{MumfordFPP} to construct the first example of a complex surface with the same Betti numbers as $\bb C\mathbb P^2$ but not biholomorphic to it. These surfaces are now called \emph{fake projective planes} (FPPs for short).

Fake projective planes have ample canonical bundle and hence are algebraic surfaces of general type by Chow's Theorem. They serve as canonical examples of complex surfaces of general type with the smallest Euler characteristics. Studying their geometry and classification is a subject of interest for many algebraic geometers.

For a fake projective plane $X$, the Hodge Theorem implies that $H^j(X, \bb C) = \bigoplus_{p+q = j} H^{p, q}(X)$. It then follows from the Hodge symmetry that the Hodge numbers of $X$ must be
\[h^{p, q}(X) = \begin{cases}
1, \text{if $(p, q) \in \{(0, 0), (1, 1), (2, 2)\}$}\\
0, \text{otherwise.}
\end{cases}\] Let $c_1$ and $c_2$ be the first and second Chern numbers of $X$ respectively, Noether's formula asserts that
\[ \frac{c_1^2 + c_2}{12} = \chi(0) = \sum_{j = 0}^2 (-1)^j h^{0, j}(X) =  1\]
Recalling that $c_2$ is equal to the topological Euler characteristics of $X$, it then follows that $c_1^2 = 9=3c_2$ and hence the Bogomolov-Miyaoka-Yau inequality is an equality. A classical result in \cite{Yau1977CalabisCA} asserts that this equality is true if and only if $X$ is the quotient of the complex $2$-ball $\bb B^2 \coloneqq \{(z, w) \in \bb C^2\ |\ |z|^2 + |w|^2 < 1\}$ by a torsion-free co-compact discrete subgroup of $PU(2, 1)$.

Much of the work in classifying all fake projective planes has been done by analyzing these discrete subgroups. In the same paper by Mumford \cite{MumfordFPP}, he noted that there exist only finitely many fake projective planes up to isomorphism. This is a consequence of Weil's result \cite{Weil1960} that discrete co-compact subgroups of $PU(2, 1)$ are rigid. Later on, Prasad and Yeung \cite{Prasad_Yeung_2007} showed that all fake projective planes must fall into one of $28$ distinct non-empty classes. Finally, Cartwright and Steger \cite{CARTWRIGHT201011} obtained a complete classification of $50$ complex conjugate pairs of fake projective planes in the aforementioned $28$ different classes. For a comprehensive survey on the history of classifying all fake projective planes, we refer the reader to the expository paper by Yeung \cite{Yeung_2008}.

In this paper, we will study the pair of fake projective planes given by $(a = 7, p = 2, \emptyset, D_3 2_7)$ and its complex conjugate in the Cartwright-Steger classification. It is one of the only $3$ pairs of fake projective planes whose automorphism group has the largest cardinality, being the unique non-abelian semi-direct product $C_7 \rtimes C_3$. The other two pairs are those of the fake projective plane $(C20, p = 2, \emptyset, D_3 2_7)$ and the Keum's fake projective plane \cite{keum2006fake} $(a = 7, p = 2, {7}, D_3 2_7)$.

\begin{rem}
Explicit equations of $(a = 7, p = 2, \emptyset, D_3 2_7)$ and its complex conjugate were computed by Borisov and Keum in \cite{Borisov_2020}. These surfaces are cut out in their bicanonical embeddings by $84$ cubic equations in ten variables, with coefficients in the ring $\mathbb Z[\sqrt{-7}]$. However, current techniques do not allow one to determine which of the choices of $\sqrt{-7}\in \mathbb C$ leads to the ball quotient with the group generators given by Cartwright and Steger for the 
 $(a = 7, p = 2, \emptyset, D_3 2_7)$ case and which leads to its complex conjugate (which simply means conjugating all of the entries of the group elements). Everything we do in this paper applies equally to both surfaces, but for the sake of definitiveness, we denote  by $\FPP$ the surface given by the equations of \cite{Borisov_2020} with $\sqrt{-7} = {\mathrm i}\sqrt 7$.
\end{rem}

It is known that the Picard group of $\FPP$ is given by $\Pic(\FPP) \cong \bb Z H \oplus (\bb Z/2\bb Z)^4$ where
$H$ is the unique ample divisor class on $\FPP$ such that $3H$ is the canonical class. The equations of \cite{Borisov_2020} describe the bicanonical embedding of $\FPP$ in $\bb C\mathbb P^9$ via $10$ global sections of $H^0(\FPP, 6H)$. 

Our first goal is to study how the automorphism group $\Aut(\FPP) \cong C_7 \rtimes C_3$ of $\FPP$ acts on $\Pic(\FPP)$. Clearly any automorphism has to fix $H$, so the question amounts to investigating what happens on the torsion subgroup. For each nontrivial torsion element $T \in \Pic(\FPP)$, we will compute explicit representatives of $3H + T$ as non-reduced (double) cuts in $6H$ in Section~\ref{sec:3H_Tor}. We focus on three torsion classes  - $D, D_1, D + D_1$ - which are the representatives of the $C_7$ orbits of the automorphism action. 

Another question we are interested in is how the linear systems $|nH + T|$ behave for $n \leq 5$ and torsion line bundles $T \in \Pic(\FPP)$. This is because as $n$ increases, the expansion in dimensions tends to trivialize phenomena in lower dimensions (for example, the Kodaira Vanishing Theorem becomes applicable for $n \geq 4$). Specifically, we show that

\begin{thm}\label{thm:main_divisor_relation}
On the fake projective plane $\FPP$ and for torsion line bundle $T \in \Pic(\FPP)$,
\begin{enumerate}
    \item \label{item:2H}$h^0(\FPP, 2H + T) = 0$.
    \item \label{item:4H}$4H + T$ is base point free if and only if $T$ is nontrivial.
    \item \label{item:5H}$5H + D$ is very ample, but $5H$ is not very ample.
\end{enumerate}
\end{thm}

\begin{rem}
Note that the case of $T \in \{0, D\}$ in Theorem~\ref{thm:main_divisor_relation}(\ref{item:2H}) is a consequence of Theorem 5.3 in \cite{Galkin_2023}. Our contribution is for the other $14$ torsion line bundles. As a consequence (or rather in the proof of) Theorem~\ref{thm:main_divisor_relation}(\ref{item:5H}), we also produce an explicit embedding of $\FPP$ as the zero set of $56$ sextics in $\bb C\mathbb P^5$ with coefficients in $\bb Q(\sqrt{-7})$.
\end{rem}

\begin{rem}
For all fake projective planes $X$ with known explicit equations (see \cite{Borisov_2020}, \cite{Borisov-Aut21}, \cite{borisov2020new}, and \cite{borisov2020journey}), their embeddings were constructed via $|6H_{X}|$, where $H_X$ is the unique ample divisor class on $X$ such that $3H_X = K_X$ the canonical divisor class. In \cite{borisov2023realizing}, the authors were able to further embed Keum's fake projective plane with $5H_X$. Whether or not $6H_{X}$ is very ample for all fake projective planes $X$ is still an open question, but Theorem~\ref{thm:main_divisor_relation}(\ref{item:5H}) shows that a similar conjecture is false for the case of $5H_{X}$. The observations that $4H$ is not base point free and $5H$ is not very ample came as a surprise to us. After all, there is no obvious difference between our example and that of Keum's fake projective plane considered in \cite{borisov2023realizing}, where this was not the case.
\end{rem}

A common feature in the research on fake projective planes is liberal use of mathematical computing. The proof of Theorem~\ref{thm:main_divisor_relation} depends heavily on the use of the computer algebra systems \texttt{Mathematica} \cite{Mathematica}, \texttt{Magma} \cite{Mag}, and \texttt{Macaulay2} \cite{M2}. Our code repository \cite{BJLM-code} accompanying this paper is available on GitHub and on our webpage.

\subsection{Outline}
In Section~\ref{sec:3H_Tor}, we first describe a method of computing $3H + D, 3H + D_1, 3H + D + D_1$ and determining the group relations of the torsion divisors. We also find the explicit quadratic polynomials vanishing on $3H + T$ for each torsion divisor $T$. In Section~\ref{sec:4H}, we compute explicit representatives of $|4H|$ and prove the ``only if" direction of Theorem~\ref{thm:main_divisor_relation}(\ref{item:4H}). In Section~\ref{sec:5H_Tor}, we use the explicit representatives of $|4H|$ to compute explicit maps given by sections of $5H$ and $5H + D$ into $\bb C\mathbb P^5$ to prove Theorem~\ref{thm:main_divisor_relation}(\ref{item:5H}), and we also compute explicit equations for the zero locus of the $C_7$-equivariant sections of $5H$. In Section~\ref{sec:4H_Tor}, we use the sections of $5H$ to compute explicit equations for the zero locus of the sections $4H + T$ for $T \neq 0$ a torsion line bundle. This proves the ``if" direction of Theorem~\ref{thm:main_divisor_relation}(\ref{item:4H}). In the same section, we also use this result to prove Theorem~\ref{thm:main_divisor_relation}(\ref{item:2H}). 

\subsection*{Acknowledgements}\label{sec:acknowledgements}
This research was conducted while the second author was participating during the 2023 DIMACS REU program at the Rutgers University Department of Mathematics, mentored by the first author. We are grateful to Lazaros Gallos, Kristen Hendricks, and the Rutgers University DIMACS for organizing the REU program and making this project possible. We thank the anonymous referee for multiple useful comments on the first version of the paper.%

\section{Computing torsion divisors with $3H + T$}\label{sec:3H_Tor}

In this section, we construct explicit non-reduced cuts of $\FPP$ representing $3H + T$ for each nontrivial torsion line bundle $T \in \Pic(\FPP)$. This is done in two steps. We first find the cuts in a finite field with the code \texttt{3H-Reduction.txt} in our code repository \cite{BJLM-code}. We then lift the solutions back to $\bb Q(\sqrt{-7})$ in the first part of the file \texttt{3H-Torsion.nb}.

In addition, we compute the group relations between the explicit representatives we found. Finally, we also compute quadratic polynomials vanishing on each curve in the class of $3H + T$ in the embedding of $\FPP$ in $\bb C\mathbb P^9$. These calculations are done in the remainder of \texttt{3H-Torsion.nb}.

\subsection{Construction of the non-reduced cuts}\label{subsec:cuts}

We would like to construct the torsion classes of the Picard group. Suppose $T$ is a nontrivial torsion class, then Theorem 2.2 of \cite{Galkin_2023} asserts that $H^0(\FPP, 3H + T)$ is one-dimensional. Suppose we take generator $\alpha \in H^0(\FPP, 3H + T)$, then since all nontrivial torsion elements of the Picard group have order $2$, we have $\alpha^{\otimes 2} \in H^0(\FPP, 6H)$. Thus, we would like to construct $3H + T$ from elements in $H^0(\FPP, 6H)$ whose zeros form a nonreduced, more specifically double, curve.

We start by examining how the automorphism group of $\FPP$ interacts with the torsion classes. In \cite{Borisov_2020}, the automorphism group $\mathrm{Aut}(\FPP)\cong C_7 \rtimes C_3$  acts on the homogeneous coordinates of $\bb C\mathbb P^9$ as follows:
\begin{multline}\label{eq:c3}
g_3(U_0:U_1:U_2:U_3:U_4:U_5:U_6:U_7:U_8:U_9) = (U_0:U_2:U_3:U_1:U_5:U_6:U_4:U_8:U_9:U_7)
\end{multline}
\begin{multline}\label{eq:c7}
g_7(U_0:U_1:U_2:U_3:U_4:U_5:U_6:U_7:U_8:U_9) \\ = (U_0: \zeta ^6U_1: \zeta ^5U_2: \zeta ^3U_3:\zeta^1 U_4:\zeta ^2U_5:\zeta ^4U_6:\zeta ^1 U_7:\zeta ^2U_8:\zeta^4U_9), 
\end{multline}
where $g_3$ and $g_7$ are generators of $C_3$ and $C_7$ respectively and $\zeta$ is the $7$-th root of unity $\exp(2\pi i/7)$. 

The linear span of $U_0,\ldots,U_9$ splits as a direct sum of the trivial representation of $\mathrm{Aut}(\FPP)$, spanned by $U_0$, and three $3$-dimensional representations of it. 
It was observed in \cite{Borisov_2020} that $U_0=0$ cuts out a double curve on $\FPP$, which therefore corresponds to a nontrivial divisor class $D$ in the torsion of $\Pic(\FPP)$ which is invariant under $\mathrm{Aut}(\FPP)$. In the other direction, since any $\mathrm{Aut}(\FPP)$-invariant nontrivial torsion element of $\Pic(\FPP)$ must give rise to a one-dimensional subrepresentation in $H^0(\FPP,6H)$, we see that $D$ is the only such element.

The order of the automorphism group is odd, so by Maschke's Theorem, we get the splitting of  $\mathrm{Aut}(\FPP)$
representations over the field $ (\bb Z/2\bb Z)$
\begin{equation}\label{eq:dsum}
{\mathrm{Torsion}}(\Pic(\FPP)) \cong (\bb Z/2\bb Z)D \oplus (\bb Z/2\bb Z)^3.
\end{equation}
The action of the automorphism $g_3$ on the set of $7$ nonzero elements of the direct summand $(\bb Z/2\bb Z)^3$ of \eqref{eq:dsum} must have at least one fixed point, and we will denote it by $E$. We know that $E$ is not invariant with respect to the entire
 $\mathrm{Aut}(\FPP)$, so all the other nonzero elements of $(\bb Z/2\bb Z)^3$ are $C_7$ translates of $E$. Thus
 their stabilizers are conjugates of $\langle g_3\rangle$ in $C_7 \rtimes C_3$, so  $E$ is the only one of them that is fixed by $g_3$. We have thus proved that there are exactly three nontrivial torsion classes 
$$
D,~E,~D+E
$$
in $\Pic(\FPP)$ that are invariant under the action of $C_3=\langle g_3\rangle$. 

Our goal is now to find these classes explicitly, by finding the corresponding sections $\alpha^{\otimes 2} \in H^0(\FPP, 6H)$. 
For a $C_3$-invariant torsion divisor class, the corresponding bicanonical section $\alpha^{\otimes 2}$ must be invariant under the action of $C_3$, up to scaling, so it must lie in an eigenspace of this action. It is reasonable to consider the $1$-eigenspace first, i.e. to look for $C_3$-invariant cuts of $\FPP$ of the form 
\begin{equation}\label{cutinv}
U_0+a_1(U_1+U_2+U_3) + a_2(U_4+U_5+U_6) + a_3(U_7 + U_8 + U_9)=0.
\end{equation}
Our approach was rather natural. We first found a solution over a finite field, then got a good enough $p$-adic approximation and then recognized the resulting coefficients $a_i$ as elements of $\mathbb Q(\sqrt{-7})$. Importantly, we did not have to be rigorous at the intermediate steps as long as we verified that the final cuts give us double curves on $\FPP$.

We know that the cut \eqref{cutinv} must give a nonreduced curve, so it will remain nonreduced after a reduction modulo a prime $q$ in the ring $\mathbb Z[\sqrt{-7}]$. We can pick a prime $p$ in $\mathbb Z$ which splits in $\mathbb Z[\sqrt{-7}]$, i.e. $(-7)$ is a square in $\mathbb{F}_p$. By picking one choice of the square root, we can realize the reduction of $\FPP$ modulo $q$ as a subscheme in $\mathbb F_p\mathbb P^9$. The possible cuts \eqref{cutinv} modulo $p$ can then be found by an exhaustive search over all possible $(a_1,a_2,a_3)\in \mathbb F_p^3$, using Magma to see if the cut is nonreduced.
Specifically, the smallest such prime for which the reduction of equations of $\FPP$ preserves the Hilbert polynomial is $p=11$, and we picked $\sqrt{-7}$ to be $2\hskip-3pt\mod 11$. Then, there were only three nonreduced cuts for $(a_1,a_2,a_3)$ given by $(0, 0, 0), (7, 0, 0)$ and $(8,7,7)$ respectively. The first triple corresponds to $U_0=0$ and $3H + D$.

For each of the last two cuts modulo $11$, our goal was to keep lifting the coefficients $a_i$ modulo higher and higher powers of $11$ while keeping the corresponding curves singular, in the appropriate sense. More precisely, we made sure that the curves were at least singular at enough points to ensure unique lifting. To begin with, we obtained several points on the last two curves. The curves had a tangent space of dimension two at each of these points, and our goal was to keep lifting the coefficients to powers of $11$ while keeping these points singular, by simultaneously lifting tangent vectors.  Specifically, suppose one of the points is $x = [x_0: ...: x_9]$. We can set without loss of generality $x_0 = 1$. Then we consider 2 tangent vectors to the curve of the form $v = (v_0, v_1, ..., v_9)$. We set similarly $v_0 = 0$, and then $v_1 = 1, v_2 = 0$ for one and $v_1 = 0, v_2 = 1$ for the other to make them linearly independent (we were somewhat lucky that we were able to do it, otherwise the code would have been a bit more complicated). To get the tangent vectors in $T_x \FPP$, we substituted in $x_i + v_i t$'s  for $U_i$'s for each of the 84 equations and the linear cut we obtained, and solved for conditions such that the coefficient of $t$ is zero, which means that the vector is orthogonal to the gradient of all these equations. We then lifted the coefficients to modulo ${11^2}$ as follows. We substituted $\sqrt{-7}$ with the appropriate value modulo $11^2$, and adjusted the linear cuts by adding $11(a_1'(U_1+U_2+U_3) + a_2'(U_4+U_5+U_6) + a_3'(U_7 + U_8 + U_9))$, for some unknown $(a_1', a_2', a_3')$, which ensures that it agrees with the original cuts$\mod{11}$, and similarly the points on these cuts and the tangent vectors at these points. Then we solved for the same equations modulo $11$ to obtain the unknown coefficients $a_i$ modulo $11^2$. 

We repeated the same process with the new data to consecutively lift them to modulo $11^d$ for larger $d$. In the end we obtained the coefficients for the cuts modulo $11^{21}$. From these we recovered the original coefficients in $\mathbb{Q}(\sqrt{-7})$ as follows.
If an element $a\in \mathbb{Q}(\sqrt{-7})$ satisfies an integer relation $c_1 + c_2 \sqrt{-7}  + c_3  a  = 0$, it leads to $c_1 +c_2[\sqrt{-7}]+c_3[a]=0$ for the corresponding $11$-adic numbers. If we have approximations $[a]_{11^{21}}$ and $[\sqrt{-7}]_{11^{21}}$ of $[a]$ and $[\sqrt{-7}]$ respectively, then we get an integer relation
$$c_1 + c_2 [\sqrt{-7}]_{11^{21}} + c_3 [a]_{11^{21}} + c_4 11^{21}=0.
$$
If $c_1$, $c_2$ and $c_3$ are small, then the above relation can be recovered by using a lattice reduction algorithm to find equations of small norm in the lattice of integer relations on $1,[\sqrt{-7}]_{11^{21}},[a]_{11^{21}}$ and $11^{21}$. 
In this way, we obtained the two cuts as
\begin{equation}\label{cut1}
U_0 + \frac{1}{2}(1 + \sqrt{-7})(U_1 + U_2 + U_3)=0
\end{equation}
and
\begin{equation}\label{cut2}
U_0 + (-5 + \sqrt{-7})(U_1 + U_2 + U_3) + (4 - 4\sqrt{-7})(U_4 + U_5 + U_6) - 4(U_7 + U_8 + U_9)=0.
\end{equation}

\subsection{Verification of the non-reduced cuts}\label{subsec:verify_cuts}

How does one verify that, say, the first of these equations is a square $\alpha^{\otimes 2}$ of the section of $3H+T$ for some torsion divisor $T$? An optimistic approach would be to add the ideal generated by the linear function in \eqref{cut1} to the ideal $I$ that gives $\FPP$ and then use Magma to compute the radical of the resulting ideal. Unfortunately, the radical calculation was prohibitively time-consuming, at least with our hardware. What worked instead was the approach of \cite{Borisov_2020} which we will now describe. We used Mathematica to find, with high numerical accuracy, a large number of random points on the putative reduced curve $C$ such that $2C$ is cut out by the linear equation \eqref{cut1}. Then we computed, numerically with high accuracy, the space of homogeneous quadratic polynomials in $U_0,\ldots,U_9$ which vanish at these points. After (numerical) row reduction, we could recognize the coefficients of these polynomials as being close to elements of the field $\mathbb Q(\sqrt{-7})$, which gave us a putative ideal $J$, with $28$ generators of degree $2$ such that $I+J$ cuts out $C$, scheme theoretically. Indeed, we verified in Magma, now with exact coefficients, that the Hilbert polynomial of (the quotient of) the ideal $I+J$ is that of a curve of the correct degree $18$. Importantly, we checked that 
$$J^2 + I,~J^2 + (U_0 + \frac{1}{2}(1 + \sqrt{-7})(U_1 + U_2 + U_3))+I,~(U_0 + \frac{1}{2}(1 + \sqrt{-7})(U_1 + U_2 + U_3))+I$$
have the same Hilbert polynomial, which means that they are equal up to saturation. This verifies that \eqref{cut1} cuts out $2C$ where $C$ is the curve given by $I+J$. Although we don't need this fact, it turns out that $I$ is contained in $J$, which can also be seen by computing Hilbert polynomials. We used the same approach to verify that \eqref{cut2} gives a double curve and to find the defining equations of that curve.

Having  verified that \eqref{cut1} and \eqref{cut2} give double curves, the $C_3$-invariance now implies that 
one of them describes the double curve for $E$ and the other describes the double curve for $D+E$. Since \eqref{cut1} looks simpler than \eqref{cut2}, it was tempting to conjecture that it would correspond to $E$, and indeed this turned out to be the case. However, additional work was required to verify it, which we will now describe. The essential difference between $E$ and $D+E$ is that the $C_7$-translates of the latter generate the whole torsion subgroup, whereas for the former we only get $(\mathbb Z/2\mathbb Z)^3$.

The fifteen nontrivial torsion elements of $\Pic(\FPP)$ correspond to $U_0=0$ and the $C_7$-translates of the equations 
in \eqref{cut1} and \eqref{cut2} under \eqref{eq:c7}. We denote the line bundles that correspond to the $(g_7)^{k-1}$-translates of the cut \eqref{cut1} by $D_k$ for $k=1,\ldots,7$ and we denote the ones that correspond to $(g_7)^{k-1}$-translates of the cut \eqref{cut2} by $D_{7+k}$. We can recover the group operation as follows. 
Suppose three distinct curves $C_1$, $C_2$ and $C_3$ on $\FPP$ are the zeros of the sections of $3H+L_1,3H+L_2,3H+L_3$ respectively, for nontrivial torsion line bundles $L_i$. Then we can look for sections of $12H$ which  vanish on the union of these curves. 
We see that $H^0(\FPP,12H - C_1-C_2-C_3)=H^0(\FPP,3H-(L_1+L_2+L_3))$ is the kernel of the restriction map from $H^0(\FPP,12H)$ to  $\bigoplus_{i=1}^3 H^0(C_i,12H)$. We know that $L_1+L_2+L_3=0$ if and only if this kernel is zero.
We don't even need to compute the entire restriction map. The space of $H^0(\FPP,12H)$ happens to be the $55$-dimensional space of quadratic polynomials in $U_0,\ldots, U_9$, so if we find enough points on $C_i$ to have the matrix of values at these points have the maximum possible rank $55$ (can be done with the interval arithmetic of Mathematica), then we know that $L_1+L_2+L_3=0$.

We did this matrix rank computation for every triple in the $15$ nontrivial torsion classes to determine the group law. The complete table is in the file \texttt{3H-Torsion.nb}, there we verify that $E$ indeed corresponds to \eqref{cut1}. We checked that 
 $D_1, ..., D_7$ can be written in terms of the basis $\{D_1, D_2, D_3\}$ as
\begin{equation}\label{eq:basis}
\begin{split}
D_4 &= D_1 + D_2, \quad D_5 = D_2 + D_3, \quad D_6 = D_1 + D_2 + D_3, \quad D_7 = D_1 + D_3
\end{split}
\end{equation}
which implies that $C_7$-translates of $D_1$ lie in a $(\mathbb Z/2\mathbb Z)^3$ subgroup of $\Pic(\FPP)$. 

\section{Computing $4H$}\label{sec:4H}

In this section, we compute the linear system $|4H|$. This will provide a proof for the ``only if" direction of Theorem~\ref{thm:main_divisor_relation}(\ref{item:4H}). We follow the general method of  \cite[Section 3.2]{borisov2023realizing} by Borisov and Lihn, where they computed the linear system $|4H|$ on Keum's fake projective plane. Our specific constructions are realized in the file \texttt{4H.nb} in our code repository \cite{BJLM-code}.

By the Riemann-Roch and the Kodaira vanishing theorems, we have $\dim_{\bb C} H^0(\FPP, 4H) = 3$. The line bundle $4H$ is preserved by automorphisms of $\FPP$, thus the action of $\mathrm{Aut}(\FPP)$  can be lifted to an action of a central extension of it on $H^0(\FPP, 4H)$. Since  $\mathrm{Aut}(\FPP)$ has trivial Schur multiplier, it can be made to act on $H^0(\FPP, 4H) $.
This space then splits into three one-dimensional $C_7$-eigenspaces, which, by the holomorphic Lefschetz fixed-point formula, have eigenvalues $\xi^3, \xi^6, $ and $\xi^5$ respectively, and $C_3$ then permutes the eigenspaces. Let $r_3, r_6, $ and $r_5$ be the sections generating each eigenspace. Because the lift of $g_3$ still has order three, we can scale $r_i$ so that the action of $g_3\in C_3$ sends $r_3 \to r_6 \to r_5 \to r_3$.

Let's define $s_i = r_i^{\otimes 3}$ for $i = 3, 6, 5$ and $d = r_3 \otimes r_6 \otimes r_5$. Note that $s_3, s_6, s_5, d \in H^0(\FPP, 12H)$ and thus can be represented as quadratic polynomials in $U_0, ..., U_9$. Since $s_3$ has $C_7$ weight $3 \times 3 \equiv 2 \Mod{7}$ and $d$ has weight $3 + 5 + 6 \equiv 0 \Mod{7}$ and is $C_3$ invariant, we can write $s_3$ and $d$ as
\begin{equation}\label{eq:s_3_and_d}
\begin{split}
s_3 &\coloneqq b_5 U_1 U_3 + b_3 U_4^2 + b_8 U_0 U_5 + b_4 U_2 U_6 + b_2 U_4 U_7 + b_1 U_7^2 +b_7 U_0 U_8 + b_6 U_2 U_9 \\
d &\coloneqq e_1 U_0^2 + e_2 (U_1 U_4 + U_2 U_5 + U_3 U_6) + e_3(U_1 U_7 + U_2 U_8 + U_3 U_9)
\end{split}
\end{equation}
Note that $s_5$ and $s_6$ may be obtained as $C_3$-translates of $s_3$.

Our goal is then to solve for $s_3$ and $d$ explicitly. Then we could solve for $\{s_3 = d = 0\}$ to obtain the section $r_3$ and use the $C_3$ action to find $r_6$ and $r_5$.

\subsection{Solving for $s_3$ and $d$}

The overall idea is that we have the identity
$$s_3 s_5 s_6 - d^3 = 0$$ 
on $\FPP$ which provides us equations on the coefficients of $s_3$ and $d$.
Specifically, we may compute random points on $\FPP$ with high accuracy and evaluate the expression $s_3 s_5 s_6 - d^3$ on these points. This will produce relations on the coefficients $b_1, ..., b_8$ and $e_1, e_2, e_3$, which we can then solve in the Magma file \texttt{4H-Quadratic.txt}.

However, the process above was taking quite long. To reduce the run-time, we also computed some additional constraints on the coefficients before passing it down to the \texttt{4H-Quadratic.txt}. The additional constraints are calculated as follows:

\begin{itemize}
    \item We observe that there are three $C_7$ fixed points on $\FPP$, in a $C_3$ orbit, given by
\begin{equation}\label{eq:c7-fixed-pts}
\begin{split}
p_1 &\coloneqq [0 : 0: 0: 0: 0: 0: 0: 1: 0: 0], \quad p_2 \coloneqq [0 : 0: 0: 0: 0: 0: 0: 0: 1: 0], \\
p_3 &\coloneqq [0 : 0: 0: 0: 0: 0: 0: 0: 0: 1]
\end{split}
\end{equation}
We note that $p_2$ must be in the curve $\{r_3 = 0\}$, because the only quadratic monomial that does not vanish on $p_2$ is $U_8^2$, which has $C_7$ weight $4 \Mod{7}$. On the other hand, we observe that $s_3 = r_3^3$ has weight $2 \Mod{7}$ and thus does not have a term $U_8^2$, so $r_3$ must vanish on $p_2$. Similarly, we also have that $p_3 \in \{r_3 = 0\}$ as $U_9^2$ has weight $1 \Mod{7}$.
    \item  It follows that  $s_3$ vanishes on $p_2$ and $p_3$ up to multiplicity $3$. We then compute the order $3$ formal neighborhoods of $p_2$ and $p_3$ respectively (in practice, we only needed them up to order $2$). Then we solve for the conditions of $s_3$ being identically zero on the formal neighborhoods of $p_2$ and $p_3$ up to order $2$. This reduces the number of independent coefficients in $s_3$ from $8$ to $6$.
\end{itemize}

After solving the relations on the remaining $9$ coefficients produced by the random points, we obtain the following solutions:
\begin{equation}\label{eq:s3}
\begin{split}
s_3 &= \frac{1}{8} \biggl(\frac{1}{29} (1- 27 \rr) U_1 U_3 + \frac{4}{29} (101- \rr) U_4^2 + \frac{8}{29} (15+\rr) U_0 U_5 + \frac{8}{29} (1+2 \rr) U_2 U_6 \\ &+ 8 U_4 U_7 - 4 U_0 U_8 + \frac{1}{29} (101- \rr) U_0 U_8 + \frac{1}{58} (101-\rr) U_2 U_9 + \frac{1}{58} \rr (101-\rr) U_2 U_9 \biggr)
\end{split}
\end{equation}
\begin{equation}\label{eq:d}
d = \frac{1}{812} \biggl((35 - 17 \sqrt{-7}) U_0^2 + (70 - 34\sqrt{-7})(U_1 U_4 + U_2 U_5 + U_3 U_6) + (21 + 13 \sqrt{-7})(U_1 U_7 + U_2 U_8 + U_3 U_9) \biggr)    
\end{equation}

\begin{rem}
While the authors of \cite{borisov2023realizing} solved the system $s_3 s_5 s_6 - d^3 = 0$ using a finite field search and lifting the coefficients back to $\bb Q(\sqrt{-7})$, we instead solve for the irreducible components of the ideal formed by the equations in the file \texttt{4H-Quadratic.txt} with Magma directly and found the exact solutions in $\bb Q(\sqrt{-7})$ within a reasonable time. 
\end{rem}

\subsection{Solving for $r_3=0$}

We found random points on $r_3=0$ by computing random solutions to $\{s_3 = d = 0\}$ on $\FPP$ with high accuracy. We then solved for the quadratic polynomials vanishing on $r_3$ and obtained $\dim_{\bb C} H^0(\FPP, 12H - 4H) = 21$ equations, which can be shown to cut out the curve in the correct class by looking at its Hilbert polynomial. The equations can be found in \texttt{Equations/4H\_one\_section.txt}. Now we can give a proof that $4H + 0$ is not base point free.

\begin{proof}[Proof of ``only if" direction of Theorem~\ref{thm:main_divisor_relation}(2)]
Equations in \texttt{Equations/4H\_one\_section.txt} have no monomials of the form $U_7^2, U_8^2, $ or $U_9^2$. This means that $p_1, p_2, p_3 \in \{r_3 = 0\}$. The $C_3$ action on $\FPP$ permutes $p_1 \to p_2 \to p_3 \to p_1$, so it follows that $r_5$ and $r_6$ also vanish on these $3$ points.
\end{proof}

\section{Computing $5H + D$ and $5H$}\label{sec:5H_Tor}
In this section, we prove Theorem~\ref{thm:main_divisor_relation}(\ref{item:5H}). For $5H + D$, we will construct an embedding of $\FPP$ in $\bb C\mathbb P^5$ given by its sections to show it is very ample in Section~\ref{subsec:5HD}.  For $5H$, we will construct the image of $\FPP$ under its map and show the image is singular in Section~\ref{subsec:5H}. We will also verify that $5H$ is base point free after we compute the quadratic polynomials vanishing on the global sections of $5H$ in Section~\ref{subsec:5H_Quadratic}. Most of the relevant computations are laid out in the Mathematica file \texttt{5H-Torsion.nb} in our code repository \cite{BJLM-code} except for the check that the sections of $5H$ do not have any common zeros, which is done in the Magma file \texttt{5H-intersection.txt}.

\subsection{For $5H + D$} \label{subsec:5HD}

With the computations of $3H + D$ in Section~\ref{sec:3H_Tor} and $4H$ in Section~\ref{sec:4H}, we can now find six linearly independent sections on $5H + D$. Observe that $12H = (5H + D) + (3H + D) + 4H$, so we can compute for linear combinations of quadratic monomials in $U_0,\ldots,U_9$ that vanish on random points from both $4H$ and $3H + D$. This produced $6$ degree two polynomials in these variables that represent the $6$ global sections of $5H + D$.

These $6$ degree two polynomials produce a map $\FPP \to \bb C\mathbb P^5$, so we can construct with high accuracy random points of the image of $\FPP$ in $\bb C\mathbb P^5$. We then solved for sextic equations in terms of the $6$ quadratic polynomials and found $56$ sextics with coefficients in $\bb Q(\sqrt{-7})$. To check that this is indeed an embedding, we follow the verification process carried out in \cite{borisov2023realizing}. The relevant verification files are in the folder \texttt{Verification/5H+D}.

\subsection{For $5H$}\label{subsec:5H}

Because $h^0(\FPP, 3H) = 0$, we can't use the method in Section~\ref{subsec:5HD} to find the sections of $5H$. Instead, we observe that $30H = 5H + 4H + \sum_{i = 1}^7 (3H + D_i)$, so we can look for linear combinations of quintic monomials vanishing on random points  on $4H, 3H + D_1, ..., 3H + D_7$. This produced $6$ quintic polynomials in variables $U_0, ..., U_9$ which we will name $Z_1, ..., Z_6$.

Let $X$ denote the image of $\FPP$ in $\bb C\mathbb P^5$ given by $Z_1, ..., Z_6$. We can compute random points on $X$ and solve for sextics in variables $Z_1, ..., Z_6$ that vanish on the image. This produced $59$ (as opposed to $56$) sextic equations with coefficients in $\bb Q(\sqrt{-7})$. Moreover, the Hilbert polynomial of the quotient of the polynomial ring in six variables by the ideal generated by the above $59$ sextics is $p(n)=\frac 12(5n-1)(5n-2) - 3$, which is less than the dimension $\frac 12 (5n-1)(5n-2)$ of $H^0(\FPP,5n H)$. By \cite[Exercise II.5.9(b)]{Hartshorne}, this implies that $5H$ is not very ample.

We investigated the image $X$ of $\FPP$ under this map further. It easy to see that the image is two-dimensional by calculating the Jacobian matrix of $Z_i$ at a random point of $\FPP$, with enough accuracy. We checked that there were no higher degree equations besides the sextic equations calculated for $X$, because these $59$ equations generate a prime ideal, see Magma file \texttt{5H-is-prime.txt}. Furthermore,
we checked that $X$ is singular at the following $3$ points:
\begin{equation}\label{eq:singular_pts}
q_1 \coloneqq [1: 0: 0: 0: 0: 0], q_2 \coloneqq [0: 1: 0: 0: 0: 0], q_3 \coloneqq [0: 0: 0: 1: 0: 0]
\end{equation}
by computing the rank of the Jacobian matrix at these points. 

\begin{rem}
We suspect that these are the only singular points of $X$ and that $X$ is not normal at these points.
Note that these points are fixed under the $C_7$ automorphism subgroup and are thus the images of the points 
$p_1,p_2,p_3$ that were the base points of $4H$. This indicates that there is perhaps some common reason that underlies both $4H$ having base points and the failure of $5H$ being very ample. However, we do not know what this reason might be.
\end{rem}

\subsection{Quadratic polynomials vanishing on $5H$}\label{subsec:5H_Quadratic}

For each quintic equation corresponding to each $Z_i$, we can find random points on $Z_i=0$ by computing the intersection of the quintic equation with random linear cuts. Out of the roots we computed, we only keep the points that are non-zero on the equations of $4H$ and $3H + D_i$ for $i = 1, ..., 7$. The points left will then be on the zero set of the section of $5H$.

We then solve for quadratic polynomials in $U_0, .., U_9$ which vanish on the remaining points and finds $h^0(\FPP, 12H - 5H) = 15$ quadratic  equations defining $Z_i$. We also check in the file \texttt{5H-intersection.txt} that all $90$ quadratic polynomials do not have any common zeros, so $5H$ is indeed base point free.%

\section{Computing $4H+T$}\label{sec:4H_Tor}
In this section, we use the results of Section~\ref{subsec:5H_Quadratic} to compute quadratic polynomials vanishing on sections of $4H + T$ for all nontrivial torsion classes $T \in \Pic(\FPP)$. These computations will also be used to prove Theorem~\ref{thm:main_divisor_relation}(\ref{item:2H}) and the ``if" direction of Theorem~\ref{thm:main_divisor_relation}(\ref{item:4H}). The relevant computations for the quadratic polynomials vanishing on $4H + T$ and the proof of Theorem~\ref{thm:main_divisor_relation}(\ref{item:2H}) can be found in the {Mathematica} file \texttt{4H-Torsion.nb} from our code repository \cite{BJLM-code}. The checks that $4H + T$ is base point free are in the {Magma} file \texttt{4H-Torsion-intersection.txt}.

It suffices for us to compute the quadratic polynomials on $4H + D$, $4H + D_1$, and $4H + D_8$, as the rest can be generated by the $C_7$ action. For $L \in \{D, D_1, D_8\}$, we observe that $12H = (4H + L) + (3H + L) + 5H$. We can compute the quadratic polynomials vanishing on $4H + L$ in two steps:

\begin{enumerate}
    \item First, we can find the $3$ sections of $4H + L$ as  quadratic polynomials in $U_0, ..., U_9$ vanishing on random points of $3H + L$ and one section of $5H$ (in our case, we chose $Z_2$).
    \item These sections are technically quadratic polynomials. Since we have computed explicit equations on sections of $5H$ in Section~\ref{subsec:5H_Quadratic} and of $3H + L$ in Section~\ref{subsec:verify_cuts}, we can take random cuts on each quadratic polynomials, only keeping the points that are non-zero on these equations. The remaining points will be random points on the sections of $4H + L$ as divisors. We can then find $h^0(\FPP, 12H - (4H + L)) = 21$ quadratic polynomials for each section.
\end{enumerate}
Note that we also only need to do this on one section of $4H + L$, as the other two can be populated by the $C_3$ action. Now we will finish the proofs of Theorem~\ref{thm:main_divisor_relation}.

\begin{proof}[Proof of ``if" direction of Theorem~\ref{thm:main_divisor_relation}(\ref{item:4H})]
    It suffices for us to check this for $4H + D, 4H + D_1, $ and $4H + D_8$. Since we have computed the quadratic polynomials vanishing on each already, we can simply check that they don't have common zeros in the {Magma} file \texttt{4H-Torsion-intersection.txt}.
\end{proof}

\begin{proof}[Proof of Theorem~\ref{thm:main_divisor_relation}(\ref{item:2H})]
The case for $2H$ and $2H + D$ is implied by Theorem 5.5 of \cite{Galkin_2023}. For the other $14$ torsion classes, it suffices for us to check this on $2H + D_1$ and $2H + D_8$ because of the $C_7$ action.

If $h^0(\FPP, 2H + D_1) \neq 0$, then let $\alpha \in H^0(\FPP, 2H + D_1)$ be some nontrivial section. Consider the section $\alpha \otimes r \in H^0(\FPP, 6H)$ where $r$ is any nontrivial section of $H^0(\FPP, 4H + D_1)$. Since the basis of $H^0(\FPP, 6H)$ is given by $U_0, ..., U_9$, this means that $\alpha \otimes r$ lies in some hyperplane in $\bb C\mathbb P^9$. However, since we have found random points on $4H + D_1$ previously in this section, a direct check in \texttt{4H-Torsion.nb} shows that there exist $10$ points of one section of $H^0(\FPP, 4H + D_1)$ whose determinant is non-zero, hence we have a contradiction. A similar check is done for $4H + D_8$ to show that $h^0(\FPP, 4H + D_8) = 0$.
\end{proof}

\begin{rem}
Since $h^0(\FPP, 2H) = 0$, we know that $h^0(\FPP, H + T) = 0$. This is because the existence of any nontrivial section in $h^0(\FPP, H + T)$ would imply $h^0(\FPP, 2H) \neq 0$ by tensoring with itself.
\end{rem}

\end{document}